\documentclass[12pt,reqno]{article}
\usepackage{amssymb,amscd,amsmath,amsthm,color}
\newtheorem{theorem}{Theorem}[section]

\newtheorem{cor}[theorem]{Corollary}

\usepackage{enumerate,amssymb}
\theoremstyle{definition}
\newtheorem{definition}[theorem]{Definition}
\newtheorem{example}[theorem]{Example}
\newtheorem{question}[theorem]{Question}
\theoremstyle{remark}
\newtheorem{remark}[theorem]{Remark}

\numberwithin{equation}{section}

\def\bA{\mathbb{A}}

\def\bM{\mathbb{M}}

\def\bB{\mathbb{B}}

\newcommand{\sym}{\mathrm{sym}}
\newcommand{\qsym}{\mathrm{qsym}}

\begin{document}
\baselineskip=15pt

\title{ Triangle inequalities for the operator symmetric modulus}

\author{ Jean-Christophe Bourin 
and  Eun-Young Lee
}

\date{ }

\maketitle

\vskip 10pt\noindent
{\small
{\bf Abstract.}  Let $|Z|_{\mathrm{sym}}=(|Z^*|+|Z|)/2$ denote the symetric modulus of $Z\in\bM_n$. The  triangle inequality for this modulus and the operator norm $\|\cdot\|_{\infty}$ fails : Teng Zhang recently pointed out a simple example in $\bM_2$ with
$\| |A+B|_{\sym}\|_{\infty}> \| |A|_{\sym}\|_{\infty}+\| |B|_{\sym}\|_{\infty}$. However,  for every symmetric (unitarily invariant) norm and any family of $m$ matrices in $\bM_n$, we have
$$ 
\left\| \, \left|\sum_{k=1}^m X_k\right|_{\sym} \right\| \le \sqrt{2} \left\| \sum_{k=1}^m \left|X_k\right|_{\sym} \right\| .
$$
This is derived from an inequality involving unitary orbits which also entails some eigenvalue estimates and logmajorisation relations. In some quite interesting special cases, $\sqrt{2}$ can be reduced to 1; in particular when the sum $\sum_{k=1}^m X_k=:S$ is an involution (which means $S^2=I$, the identity).

}

\section{Introduction}

Let $\bM_n$ denote the space of complex $n\times n$ matrices. $|A|=(A^*A)^{1/2}$ denotes the absolute value, or right modulus,   of $A\in\bM_n$. 
We may  read  $\bM_n$ as the space of real matrices; in this case, a unitary matrix then means an orthogonal one.

The symmetric modulus and the quadratic symmetric modulus of $Z\in\bM_n$ are respectively defined as
$$
|Z|_{\sym}:= \frac{|Z|+|Z^*|}{2}, \qquad
|Z|_{\qsym}:=\sqrt{\frac{|Z|^2+|Z^*|^2}{2}}.
$$

Very revently Teng Zhang obtained a superb  Thompson's type triangle  inequality for the quadratic modulus (\cite[Theorem 1.12]{TZsym}) :

\vskip 5pt\noindent
\begin{theorem}\label{thZ} Let $A,B\in\bM_n$. Then, for some unitaries $U,V\in\bM_n$,
$$
|A+B|_{\qsym} \le U|A|_{\qsym}U^* +V|B|_{\qsym}V^*.
$$
\end{theorem}

\vskip 5pt
Let ${\mathrm{Re\,}}Z=(Z+Z^*)/2$ and ${\mathrm{Im\,}}Z=(Z-Z^*)/2i$ denote the real and imaginary parts of $Z\in\bM_n$.
 Theorem \ref{thZ}  extends, and was motivated, by the following two matrix versions (\cite{BLsym}) of the  inequality for a complex number $z=a+ib$, $|z|\le |a|+|b|$.

\vskip 5pt
\begin{cor}\label{symq} Let $Z\in\bM_n$. Then, for some unitaries $U,V\in\bM_n$,
$$
|Z|_{\qsym} \le U |{\mathrm{Re\,}}Z| U^* + V |{\mathrm{Im\,}}Z|V^*.
$$
\end{cor}

\vskip 5pt
\begin{cor}\label{symc} Let $Z\in\bM_n$. Then, for all $j,k=0,1,2,\ldots$, 
$$
\mu_{1+j+k}^{\downarrow}\left(|Z|_{\qsym}\right) \le \mu_{1+j}^{\downarrow}({\mathrm{|Re\,}}Z|) + \mu_{1+k}^{\downarrow}(|{\mathrm{Im\,}}Z|).
$$
\end{cor}

Here
$\lambda_j^{\downarrow}(|X|)$, $j=1,2,\ldots$, denote the singular values of $X\in\bM_n$ arranged  in the  nonincreasing order,  with convention  $\lambda_k^{\downarrow}(|X|)=0$ for $k>n$. The second corollary follows from the first one by a straightforward application of the classical Weyl inequalities.

For the optimality of these two corollaries, we refer to the discussion in [5].
Since $\sqrt{t}$ is operator concave, Corollary  \ref{symq} is stronger than the following one \cite{BLsym} :

\vskip 5pt
\begin{cor}\label{syms} Let $Z\in\bM_n$. Then, for some unitary $U,V\in\bM_n$,
$$
|Z|_{\sym} \le U |{\mathrm{Re\,}}Z| U^* + V |{\mathrm{Im\,}}Z|V^*.
$$
\end{cor}

Despite of this special case,  Theorem \ref{thZ} does not hold for the classical symmetric modulus $|\cdot|_{\sym}$.  Zhang \cite{TZsym} gave a simple example with two
matrices $A,B\in\bM_2$ such that $$\| |A+B|_{\sym}\|_{\infty}> \| |A|_{\sym}\|_{\infty}+\| |B|_{\sym}\|_{\infty},$$
where $\|\cdot\|_{\infty}$ stands for the operator norm. 

This paper gives some substitutes to  Theorem \ref{thZ} for the  symmetric modulus $|\cdot|_{\sym}$. Section 2 states a main inequality  via unitary orbits.  A few corollaries follow, these are triangle inequalities for the classical symmetric modulus.  

In section 3 we prove our main inequality   and obtain further corollaries in case of Hermitian sums. This lead to introduce the class of {\it polar Hermitian matrices},  containing both involutory  matrices and Hermitian matrices.

The short Section 4 deals with   a related triangle inequality for  normal matrices $A,B\in \bM_n$ :
$$
|A+B| \le |A|+|B| +\frac{1}{4}V|A+B|V^*
$$
for some unitary matrix $V$. This inequality  is essentially well-known, the novelty is the proof of the sharpness of the constant $1/4$ when $n\ge 3$.

 Section 5 proposes  a simple proof of Theorem \ref{thZ} and shows the links between triangle inequalities and subbaditivity inequalities for concave  functions.

 Section 6 contains several comments and shows how some  our results can be improved  via the matrix geometric mean and logmajorisation.

\section{The  symmetric modulus}

\begin{theorem}\label{thsym} Let $X_1,\dots, X_m\in\bM_n$. Then, there exists a unitary $V\in\bM_n$ such that, for any $\beta>0$,
$$
\left|\sum_{k=1}^m X_k\right|_{\sym} \le \beta\left\{\sum_{k=1}^m \left|X_k\right|_{\sym}\right\} +\frac{1}{8\beta}\left\{ V\left(\sum_{k=1}^m \left|X_k\right|\right) V^* + V^*\left(\sum_{k=1}^m \left|X^*_k\right|\right) V\right\}.
$$
\end{theorem}

\vskip 5pt
We will prove it in the next section. Here, we establish some corollaries. First, with $\beta=1/2$, we have :

\vskip 5pt
\begin{cor}  Let $X_1,\dots, X_m\in\bM_n$. Then for some unitary $V\in\bM_n$,
$$
\left|\sum_{k=1}^m X_k\right|_{\sym} \le \sum_{k=1}^m \frac{\left|X_k\right| +\left|X^*_k\right| + V\left|X_k\right| V^* +V^*\left|X^*_k\right| V}{4}.
$$
\end{cor}

\vskip 5pt
Taking traces, we recapture the triangle inequality  : ${\mathrm{Tr}}\, |X+Y| \le {\mathrm{Tr}}\, |X|+{\mathrm{Tr}}\, |Y| $.
Next, 
Since $|X_k|\le 2|X_k|_{\sym}$ and $|X^*|\le 2|X_k|_{\sym}$, we get from Theorem \ref{thsym} :

\vskip 5pt
\begin{cor}\label{corsym}  Let $X_1,\cdots, X_m\in\bM_n$. Then for some unitary $V\in\bM_n$,
$$
\left|\sum_{k=1}^m X_k\right|_{\sym} \le \beta\left\{\sum_{k=1}^m \left|X_k\right|_{\sym}\right\} +\frac{1}{4\beta}\left\{ V\left(\sum_{k=1}^m \left|X_k\right|_{\sym}\right) V^* + V^*\left(\sum_{k=1}^m \left|X_k\right|_{\sym}\right) V\right\}.
$$
\end{cor}

\vskip 5pt
An application of Weyl's inequalities yields, for all integers $j>0$ and all positive matrices $S,T,R\in\bM_n$, 
$$
\lambda^{\downarrow}_{1+3j}(S+T+R) \le \lambda^{\downarrow}_{1+j}(S) +  \lambda^{\downarrow}_{1+j}(T) + 
 \lambda^{\downarrow}_{1+j}(R).
$$
Therefore, we have from Corollary \ref{corsym},
$$
\lambda^{\downarrow}_{1+3j}\left(\left|\sum_{k=1}^m X_k\right|_{\sym}\right) \le \beta\lambda^{\downarrow}_{1+j}\left(\sum_{k=1}^m \left|X_k\right|_{\sym}\right) + \frac{1}{2\beta}\lambda^{\downarrow}_{1+j}\left(\sum_{k=1}^m \left|X_k\right|_{\sym}\right).
$$
Now, the convex function $\beta\to \beta +\frac{1}{2\beta}$ is minimized on $(0,\infty)$ when its derivative vanishes, i.e.,   $\beta=1/\sqrt{2}$, and then takes its minimal value $\sqrt{2}$.
We thus obtain the bound :

\vskip 5pt
\begin{cor}\label{coreig}  Let $X_1,\dots, X_m\in\bM_n$. Then, for all integers $j=0,1,\ldots$,
$$
\lambda^{\downarrow}_{1+3j}\left(\left|\sum_{k=1}^m X_k\right|_{\sym}\right) \le \sqrt{2} \lambda^{\downarrow}_{1+j}\left(\sum_{k=1}^m \left|X_k\right|_{\sym}\right).
$$
\end{cor}

\vskip 5pt Applying the triangle inequality for any symmetric norm $\|\cdot\|$ (i.e a unitarily invariant norm) instead of Weyl's inequalities, Corollary \ref{corsym} gives in the same way  :

\vskip 5pt
\begin{cor}\label{cornor}   Let $X_1,\dots, X_m\in\bM_n$. Then, for all symmetric norms,
$$
\left\|\left|\sum_{k=1}^m X_k\right|_{\sym}\right\| \le \sqrt{2} \left\|\sum_{k=1}^m \left|X_k\right|_{\sym}\right\|.
$$
\end{cor}

\vskip 5pt
\begin{question} Is $\sqrt{2}$ the best possible constant in Corollaries \ref{coreig}-\ref{cornor} ?
\end{question}

\section{Proof of Theorem \ref{thsym} and polar Hermitians}

We prove Theorem \ref{thsym} for two matrices $X,Y\in\bM_n$, the proof for $m$ matrices being exactly the same. So we must show that there exists a unitary $V\in\bM_n$ such that for all $\beta>0$,
\begin{equation}\label{eqpr}
|X+Y|_{\sym} \le \beta\left(|X|_{\sym} +|Y|_{\sym}\right) +\frac{ V(|X|+|Y|)V^* +V^*(|X^*|+|Y^*|)V}{8\beta}.
\end{equation}

\begin{proof} The polar decomposition $Z=|Z^*|^{1/2}U|Z|^{1/2}$ shows that the following four block matrices are positive semi-definite :
$$
\begin{bmatrix} |X|&X^* \\ X&|X^*|
\end{bmatrix}, \quad \begin{bmatrix} |X^*|&X \\ X^*&|X|
\end{bmatrix}, \quad 
\begin{bmatrix} |Y|&Y^* \\ Y&|Y^*|
\end{bmatrix}, \quad \begin{bmatrix} |Y^*|&Y \\ Y^*&|Y|
\end{bmatrix}.
$$
So, adding the first one with the third one, and the second one with the fourth one,  we have two positive semi-definite matrices,
\begin{equation}\label{eqbl0}
\begin{bmatrix} |X|+|Y|&X^*+Y^* \\ X+Y&|X^*| +|Y^*|
\end{bmatrix}, \quad 
\begin{bmatrix} |X^*|+|Y^*|&X+Y \\ X^*+Y^*&|X| +|Y|
\end{bmatrix}.
\end{equation}
Now, consider the polar decompositions $X+Y=V|X+Y|$ and $X^*+Y^*=V^*|X^*+Y^*|$  and perform two unitary conjugations; one for the first matrix in \eqref{eqbl0},
\begin{equation}\label{eqbl1}
\begin{bmatrix} I&0 \\ 0&V^* 
\end{bmatrix}
\begin{bmatrix} |X|+|Y|&X^*+Y^* \\ X+Y&|X^*| +|Y^*|
\end{bmatrix}
\begin{bmatrix} I&0 \\ 0&V 
\end{bmatrix}=
\begin{bmatrix} |X|+|Y|&|X+Y| \\ |X+Y|&V^*(|X^*| +|Y^*|)V
\end{bmatrix},
\end{equation}
and the other one for the second  matrix in \eqref{eqbl0},
\begin{equation}\label{eqbl2}
\begin{bmatrix} I&0 \\ 0&V 
\end{bmatrix}
\begin{bmatrix} |X|^*+|Y^*|&X+Y \\ X^*+Y^*&|X| +|Y|
\end{bmatrix}
\begin{bmatrix} I&0 \\ 0&V^* 
\end{bmatrix}=
\begin{bmatrix} |X^*|+|Y^*|&|X^*+Y^*| \\ |X^*+Y^*|&V(|X| +|Y|)V^*
\end{bmatrix}.
\end{equation}
Next we add the right hand sides in  \eqref{eqbl1} and \eqref{eqbl2} and obtain :
$$
2\bA:=\begin{bmatrix}|X|+|Y|+ |X^*|+|Y^*|&|X+Y|+|X^*+Y^*| \\ |X+Y|+ |X^*+Y^*|& V^*(|X^*| +|Y^*|)V+V(|X| +|Y|)V^*
\end{bmatrix},
$$
hence
\begin{equation}\label{eqbl}
\bA=\begin{bmatrix} |X|_{\sym} +|Y|_{\sym} &|X+Y|_{\sym} \\
|X+Y|_{\sym}&\left\{(V(|X| +|Y|)V^*+ V^*(|X^*| +|Y^*|)V\right\}/2
\end{bmatrix}.
\end{equation}
Left and right multiplication of $\bA$ by
$$
\begin{bmatrix} I&0 \\ 0&I/2
\end{bmatrix},
$$
where $I$ is the identity of size $n$,
yields the positive matrix
$$
\bB:=\begin{bmatrix} |X|_{\sym} +|Y|_{\sym} &(|X+Y|_{\sym})/2 \\
(|X+Y|_{\sym})/2&\left\{(V(|X| +|Y|)V^*+ V^*(|X^*| +|Y^*|)V\right\}/8
\end{bmatrix}.
$$
Now, we observe that the positivity of
$$
\begin{bmatrix} \sqrt{\beta}I & - (\sqrt{\beta})^{-1} I\end{bmatrix} \bB \begin{bmatrix} \sqrt{\beta} I \\ - (\sqrt{\beta})^{-1} I\end{bmatrix}
$$
gives the desired inequality \eqref{eqpr}.
\end{proof}

\vskip 5pt
\begin{remark}\label{remproof} The proof shows that the unitary operator $V$ occurring in Theorem \ref{thsym} is the unitary factor in the polar decomposition
$$
\sum_{k=1}^m X_k = V\left| \sum_{k=1}^m X_k\right|.
$$
\end{remark}

\vskip 5pt
\begin{definition} We say that a matrix $S$ is  {\it polar Hermitian} if it has a polar decomposition $S=V|S|$ in which $V$ is a scalar multiple of a Hermitian unitary. 
\end{definition}

\vskip 5pt
\begin{example} If $S\in\bM_n$ is an involution (a symmetry, $S=S^{-1}$) then S is a polar Hermitian matrix. In fact the unitary part  of $S$ is Hermitian. This was noted in \cite{Ik}.
Here is   a  simple proof from \cite{BLinvol}. Since $S=U|S|=|S^*|U$ and $S=S^{-1}$ we infer $U|S|=U^*|S^*|^{-1}$ and so $U=U^*$ as the polar decomposition of an invertible matrix is unique. The result \cite[Theorem 2.1]{BLinvol} says much more on the structure of involutions. 
\end{example}

\vskip 5pt
\begin{theorem}\label{thsymh} Let $X_1,\dots, X_m\in\bM_n$ be such that their sum is polar Hermitian.  Then, there exists a Hermitian  unitary $V\in\bM_n$ such that, for any $\beta>0$,
$$
\left|\sum_{k=1}^m X_k\right|_{\sym} \le \beta\left\{\sum_{k=1}^m \left|X_k\right|_{\sym}\right\} +\frac{1}{4\beta}\left\{ V\left(\sum_{k=1}^m \left|X_k\right|_{\sym}\right) V\right\}.
$$
\end{theorem}

\vskip 5pt
\begin{proof}   $S=\sum_{k=1}^m X_k$    has  a  polar decomposition $S=e^{i\theta} V|S|$ for some
 Hermitian unitary $V$.  Applying Theorem \ref{thsym} combined with Remark \ref{remproof} completes the proof.
\end{proof}

\vskip 5pt
\begin{cor}\label{coreigh}  Let $X_1,\dots, X_m\in\bM_n$ be such that their sum is polar Hermitian. Then, for all integers $j=0,1,\ldots$,
$$
\lambda^{\downarrow}_{1+2j}\left(\left|\sum_{k=1}^m X_k\right|_{\sym}\right) \le \lambda^{\downarrow}_{1+j}\left(\sum_{k=1}^m \left|X_k\right|_{\sym}\right).
$$
\end{cor}

\vskip 5pt
\begin{proof} Pick $\beta=1/2$ in Theorem \ref{thsymh} and use Weyl's inequalities.
\end{proof}

\vskip 5pt
Letting $i=1+j$ with $1+2j=2i-1\le n$, we infer :

\vskip 5pt
\begin{cor}\label{coreigh}  Let $X_1,\dots, X_m\in\bM_n$ be such that their sum is a Hermitian unitary. Then, for all integers $1\le i\le (n+1)/2$,
$$
1 \le \lambda^{\downarrow}_{i}\left(\sum_{k=1}^m \left|X_k\right|_{\sym}\right).
$$
\end{cor}

\vskip 5pt
\begin{cor}\label{cornorh}   Let $X_1,\dots, X_m\in\bM_n$  be such that their sum is polar Hermitian.  Then, for all symmetric norms,
$$
\left\| \left|\sum_{k=1}^m X_k\right|_{\sym}\right\| \le  \left\|\sum_{k=1}^m \left|X_k\right|_{\sym}\right\|.
$$
\end{cor}

\vskip 5pt
\begin{proof} Pick $\beta=1/2$ in Theorem  \ref{thsymh}.
\end{proof}

\section{A sharp  inequality for normal matrices}

The  proof of Theorem \ref{thsym} is inspired by that one of \cite[Theorem 2.2]{BLposlin}. We state a part of this result in the next theorem.

\vskip 5pt
\begin{theorem}\label{thp}     Let $\Phi: \bM_n\to \bM_m$, $n,m\neq 1$ be a  positive linear map and let $N\in\bM_n$ be normal. Fix $\beta> 0$. Then, there exists a unitary $V\in\bM_m$ such that
$$
   |\Phi(N)|  \le \beta\Phi(|N|) +\frac{1}{4\beta}V\Phi(|N|)V^*.
$$
If $\beta\ge 1/2$, then the constant $1/4$ is the smallest possible one and this inequality is sharp even if we confine $N$ to the class of Hermitian symmetries in $\bM_n$.
\end{theorem}

This entails several elegant sharp inequalities. For instance, for two normal matrices $A,B\in\bM_n$, their Schur product satisfies
\begin{equation}\label{eqSchur}
|A\circ B| \le |A|\circ|B| +\frac{1}{4} V(|A|\circ|B|)V^*
\end{equation}
for some unitary $V\in\bM_n$. One easily gives an example in $\bM_2$ showing that the constant $1/4$ cannot be diminished, see \cite[Corollary 3.5]{BLposlin}.

Strangely enough we did not explicitly mention the following version for the sum of two normal matrices.

\vskip 5pt
\begin{cor}\label{cornorm} Let $A,B\in\bM_n$ be normal. Then, for some unitary $V\in\bM_n$,
$$
|A+ B| \le |A|+|B| +\frac{1}{4} V(|A|+|B|)V^*.
$$
If $n\ge 3$, the constant $1/4$ is the smallest possible one.
\end{cor}

\vskip 5pt
This inequality follows from Theorem \ref{thp} with $\beta=1$, $N=A\oplus B$, and the map $\Phi:\bM_{2n}\to\bM_n$,
$$
\begin{bmatrix}X & Y \\ Z& R\end{bmatrix} \mapsto X+R.
$$
However, contrarily to the Schur version  \eqref{eqSchur}, it seems necessary to request $n\ge 3$ in Corollary \ref{cornorm}
to get sharpness of $1/4$ (what is the best constant for $n=2$ ?). At the end of the paper \cite{BLdiag}, an example of two Hermitian contractions $A,B\in\bM_3$ is given such that
$$
|A+B| \le \frac{1}{2}I + |A|+|B|,
$$
where the constant $1/2$ cannot be diminished. Since for two contractions we have for any unitary $V$,
$$
\frac{1}{4}V(|A|+|B|)V^* \le \frac{1}{2}I
$$
we infer that the constant $1/4$ is sharp in Corollary \ref{cornorm} when $n\ge 3$.

\section{Around Zhang's triangle inequality}

We propose an approach of Theorem \ref{thZ} slighly different to the original proof by Zhang. As in Zhang's proof, the main tool is
 Thompson's inequality \cite{T}: {\it if $A,B\in\bM_n$, then there exists a pair of unitaries $U,V\in\bM_n$ such that}
\begin{equation}\label{eqthomp1}
|A+B| \le U|A|U^* +V|B|V^*.
\end{equation}
This is equivalent to the existence of a pair of contractions $K,L\in\bM_n$ such that
\begin{equation}\label{eqthomp2}
|A+B| \le K|A|K^* +L|B|L^*.
\end{equation}
Clearly $\eqref{eqthomp1} \Rightarrow  \eqref{eqthomp2}$. Conversely,  $K|A|K^*=XX^*$ with $X=K|A|^{1/2}$. So $K|A|K^*=U|A|^{1/2}K^*K|A|^{1/2}U^*$ for some unitary $U$. Since $K^*X\le I$, we have $K|A|K^*\le U|A|U^*$. Similarly, $L|B|L^*\le V|B|V^*$ for some unitary $V$. Hence  $\eqref{eqthomp2} \Rightarrow  \eqref{eqthomp1}$.

The equivalence between \eqref{eqthomp1} and \eqref{eqthomp2} is quite useful. Here is an example. Let
$X,Y\in\bM_n^+$, the positive semidefinite part of $\bM_n$. Let
$$
A=\begin{bmatrix} X&0 \\ 0&0\end{bmatrix}, \qquad B=\begin{bmatrix} 0&0 \\ Y&0\end{bmatrix}
$$
Then \eqref{eqthomp1} yields two unitary matrices $U_0,V_0\in\bM_{2n}$ such that
$$
\begin{bmatrix} \sqrt{X^2+Y^2}&0 \\ 0&0\end{bmatrix} \le  U_0\begin{bmatrix} X&0 \\ 0&0\end{bmatrix}U_0^*
+ V_0\begin{bmatrix} Y&0 \\ 0&0\end{bmatrix}V_0^*.
$$
Hence, for two contractions $K,L\in\bM_n$, we have
$ \sqrt{X^2+Y^2} \le K XK^* +LYL^*$, and so 
$$
\sqrt{X^2+Y^2} \le U XU^* +VYV^*$$
for some unitaries $U,V\in\bM_n$. This observation was one of the motivation which leaded the authors of \cite{AB} to their subbaditivity inequality : {\it If $f:[0,\infty)\to [0,\infty)$ is concave and $S,T\le \bM_n^+$, then
\begin{equation}\label{eqAB}
f(S+T)\le Uf(S)U^* +Vf(T)V^*
\end{equation}
for some unitaries $U,V\in\bM_n$.}

Now, starting with arbitrary matrices $A_1,A_2,B_1,B_2\in\bM_n$ and
$$
A=\begin{bmatrix} A_1&0 \\ A_2&0\end{bmatrix}, \qquad B=\begin{bmatrix} B_1&0 \\ B_2&0\end{bmatrix}
$$
we obtain as in the previous paragraph
$$
\sqrt{|A_1 + B_1|^2 +|A_2+B_2|^2}\le U\sqrt{|A_1|^2 + |A_2|^2}U^* + V\sqrt{|B_1|^2 + |B_2|^2}V^* 
$$
for some unitaries $U,V\in\bM_n$. Following Zhang, we pick $A_2=A_1^*$ and $B_2=B_1^*$ to get
$$
\sqrt{|A_1 + B_1|^2 +|(A_1+B_1)^*|^2}\le U\sqrt{|A_1|^2 + |A_1^*|^2}U^* + V\sqrt{|B_1|^2 + |B_1^*|^2}V^*.
$$
Dividing by $\sqrt{2}$,
$$
|A_1+B_1|_{\qsym} \le U |A_1|_{\qsym} U^* + V|B_1|_{\qsym}V^*.
$$
This is Zhang's Theorem \ref{thZ}.

One may combine \eqref{eqAB} with Zhang's theorem to get some exotic inequalities. For instance :

\vskip 5pt\noindent
\begin{cor} Let $A,B\in\bM_n$. The there exist some unitaries $U,V\in\bM_n$ such that
$$
U e^{-|A|_{\qsym}} U^* + Ve^{-|A|_{\qsym}} V^*  \le I +e^{-|A+B|_{\qsym}} 
$$
\end{cor}

\vskip 5pt\noindent
\begin{proof} Since $f(t)=1-e^{-t}$ is  increasing on $|0,\infty)$, for any pair $X,Y\in\bM_n^+$ such that $X\le Y$, we have    $X\le Wf(Y)W^*$ for some unitary $W\in\bM_n$. Hence Theorem \ref{thZ} shows that
$$
f(|A+B|_{\qsym}) \le Wf\left(U_0|A|_{\qsym}U_0^*+V_0|B|_{\qsym}V_0^*\right)W^*
$$
Applying \eqref{eqAB} to the nonnegative concave function $f(t)$, we 
get 
$$
f(|A+B|_{\qsym}) \le Uf(|A|_{\qsym})U^*+Vf(|B|_{\qsym})V^*
$$
for some unitaries $U,V$. Since  $f(Z)=I-e^{-Z}$ for every $Z\in\bM_n^+$, we get the result.
\end{proof}

\section{Concluding remarks}

This note deals with operator triangle inequalities. It is the continuation of \cite{BLsym} and a short complement to the important contribution \cite{TZsym}, running from  triangle inequalities to operator Euler identities and related Clarkson-McCarthy type inequalities. 

Though we confine to the matrix setting, we could consider versions for operators on Hilbert spaces,   or for measurable operators affiliated to a von Neumann algebra, see \cite{BS} for versions of Theorem \ref{thp} in these settings.

It remarkable that  Thompson's inequality  fo the usual operator modulus has been only very recently generalized by Zhang to the quadratic symmetric modulus. These results and our contributions for the standard symmetric modulus will  be completed with other results : there is no doubt that the noncommutative world offers a lot of perspectives on various  triangle inequalities

There are other, still very recent, known type of operator triangle inequalities.  Tang  and  Zhang \cite{LZ} prove the sharp inequality for the Hilbert-Schmidt (Frobenius) norm  : {\it If $A_1,\ldots A_m\in\bM_n$, then}
\begin{equation}\label{eqc2}
\left\| \sum_{k=1}^m A_k \right\|_2 \le  \sqrt{\frac{1 +\sqrt{m}}{2}} \left\|  \sum_{k=1}^m |A_k| \right\|_2.
\end{equation}
The case $m=2$ is the  conjecture from \cite{eyL}; it was proved by Lin and Zhang in \cite{LZ}, and another proof can be found in \cite{TZconj}. What are the symmetric moduli versions of \eqref{eqc2} ? What about the other Schatten $p$-norms ?

A geometric mean version of Theorem \ref{thp}  exists (see \cite{BLposlin}) and we  also have similar refinements of Theorems \ref{thsym}-\ref{thsymh}. For instance, for Theorem \ref{thsymh}  :

\vskip 5pt
\begin{theorem}\label{thsymhg} Let $X_1,\dots, X_m\in\bM_n$ be such that their sum is polar Hermitian.  Then, there exists a Hermitian  unitary $W\in\bM_n$ such that
$$
\left|\sum_{k=1}^m X_k\right|_{\sym} \le \left\{\sum_{k=1}^m \left|X_k\right|_{\sym}\right\} \#\left\{ W\left(\sum_{k=1}^m \left|X_k\right|_{\sym}\right) W\right\}.
$$
\end{theorem}

We did not focus on these inequalities involving the geometric mean $\#$ as we are mostly interested in triangle inequalities of the form of  Corollary \ref{cornorm} or \eqref{eqc2} with sharp constants. However Theorem \ref{thsymhg} yields an improvement of Corollary \ref{cornorh} as a weak log-majorisation :

\vskip 5pt
\begin{cor}\label{corlog} Let $X_1,\dots, X_m\in\bM_n$ be such that their sum is polar Hermitian.  Then,
$$
\left|\sum_{k=1}^m X_k\right|_{\sym} \prec_{\mathrm{wlog} }\sum_{k=1}^m \left|X_k\right|_{\sym}.
$$
\end{cor}

For a background on the geometric mean, see \cite{Ando}, \cite[Capter 4]{Bh}. For  log-majorisations relations such as $ \prec_{\mathrm{wlog} }$, see \cite{BLtv}, where a series of log-majorisation is given, extending the triangle type relation
$$
| X+Y|  \prec_{\mathrm{wlog} } |X| + |Y|
$$
for two normal matrices $X,Y\in\bM_n$.
 Corollary \ref{corlog} is immediate from Theorem \ref{thsymhg}; note that the positivity  of $\bA'$ also directly entails Corollary \ref{corlog}. 

To prove Theorem \ref{thsymhg} it suffices to consider the case of two matrices $X,Y$ whose sum is polar Hermitian. Then going back to \eqref{eqbl} and using $V=e^{i\theta}W$, $V^*=e^{-i\theta}W$, for some Hermitian unitary $W$, we have a positive block matrix
\begin{equation*}\label{eqbl}
\bA'=\begin{bmatrix} |X|_{\sym} +|Y|_{\sym} &|X+Y|_{\sym} \\
|X+Y|_{\sym}& W(|X|_{\sym} +|Y|_{\sym})W
\end{bmatrix}.
\end{equation*}
Using the maximal property of $\#$ then establishes Theorem  \ref{thsymhg}.

A special case of Corollary \ref{coreigh} is :  {\it Let $X_1,\dots, X_m\in\bM_7$ be such that their sum is a Hermitian unitary. Then,
$$
\lambda^{\downarrow}_{4}\left(\sum_{k=1}^m \left|X_k\right|_{\sym}\right) \ge 1.
$$
}
Hence decomposing a unitary Hermitian  as a sum, yields some strange constraints on the summands. This suggests the following question :

\begin{question} Let $X$ belonging either to $\bM_{2n}$ or $\bM_{2n-1}$. Suppose that $X$ is expansive, $|X|\ge I$, and consider a decomposition $X=X_1+\cdots +X_m$. What can we say about 
$$
\lambda^{\downarrow}_{n}\left(\sum_{k=1}^m \left|X_k\right|_{\sym}\right), \qquad  \lambda^{\downarrow}_{n}\left(\sum_{k=1}^m \left|X_k\right|\right)\quad  ?
$$ 
\end{question}

\vskip 15pt
\noindent
Jean-Christophe Bourin

\noindent
Université Marie et Louis Pasteur, CNRS, LmB (UMR 6623), F-25000 Besançon, France.

\noindent

\noindent

\noindent
Email: jcbourin@univ-fcomte.fr

  \vskip 15pt\noindent
Eun-Young Lee

\noindent
 Department of mathematics, KNU-Center for Nonlinear Dynamics,

\noindent
Kyungpook National University,

\noindent
 Daegu 702-701, Korea.

\noindent
  Email: eylee89@knu.ac.kr

\end{document}